\documentclass[12pt]{amsart}
\usepackage{fullpage}
\usepackage{latexsym}
\usepackage{amssymb}
\usepackage[all]{xy}

\newtheorem{thm}{Theorem}[section]
\newtheorem{prop}[thm]{Proposition}
\newtheorem{cor}[thm]{Corollary}

\newtheorem{lem}[thm]{Lemma}

\newtheorem{prob}[thm]{Problem}

\theoremstyle{definition}

\newtheorem{ex}[thm]{Example}

\usepackage{hyperref}

\newcommand{\N}{\mathbb{N}}
\newcommand{\Q}{\mathbb{Q}}
\newcommand{\R}{\mathbb{R}}
\newcommand{\Z}{\mathbb{Z}}

\newcommand{\SO}{\mathrm{\SO}}

\title[A topological interpretation of numbers]
{A topological interpretation of numbers}

\author{Christoforos Neofytidis}
\address{Department of Mathematics and Statistics, University of Cyprus}
\email{neofytidis.christoforos@ucy.ac.cy}

\date{\today}
\subjclass[2010]{55M25}
\keywords{mapping degree, realization problem, aspherical manifold, circle bundle}

\begin{document}

\maketitle

\begin{abstract} 
We survey recent advancements on the realisation problem for mapping degree sets. In particular, we explain that any finite set containing zero is the mapping degree set between some $n$-manifolds, for each $n\geq3$. Our building factors for the construction of the realising manifolds will be aspherical manifolds. Thus, along the way, we review a couple of open questions related to maps of non-zero degree for aspherical manifolds.
\end{abstract}

\section{Introduction} 

Let $M,N$ be two closed oriented topological $n$-manifolds.
The {\em degree}  of a continuous map $f\colon M\to N$, denoted by $\deg(f)$, is the integer given by $H_n(f)([M])=\deg(f)[N]$, where $H_n(f)$ is the induced homomorphism in top degree homology and $[X]$ denotes the fundamental class. The {\em mapping degree set}\index{mapping degree}\index{mapping degree!set} or {\em set of degrees of maps} from $M$ to $N$ is a subset of the integers given by
\[
D(M,N):=\{d\in\Z \ | \ \exists \ f\colon M\to N, \ \deg(f)=d\}.
\]

The purpose of this survey is to give a concise account of the following realisation problem:

\begin{prob}[Realisation Problem]\label{p:realization}
Given a set $A\subseteq\Z$, are there $M$ and $N$ such that $D(M,N)=A$?
\end{prob}

While one quickly observes that such a set $A$ must contain zero, Problem \ref{p:realization}\index{problem!realization}\index{realization!problem} is highly non-trivial. In this survey, we will outline some recent advancements, including a complete solution in the case of finite sets. From now on, we assume that any $A$ as in the Realisation Problem \ref{p:realization} contains zero.

\begin{thm}\label{t:main}\
\begin{itemize}
\item[(a)] \cite{NWW} Not all sets are mapping degree sets.
\item[(b)] \cite{NWW} All finite arithmetic sequences are mapping degree sets between $3$-manifolds.\index{realization!finite arithmetic sequence}
\item[(c)] \cite{CMV} All finite sets are mapping degree sets between $3$-manifolds.\index{realization!finite set between $3$-manifolds}
\item[(d)] \cite{NSTWW1,NSTWW2} All finite sets are mapping degree sets between $n$-manifolds, for each $n\geq4$.\index{realization!finite set between $n$-manifolds, $n\geq4$}
\end{itemize}
\end{thm}

As we shall see, part (a) of Theorem \ref{t:main} follows from a quick enumeration argument. Part (b) -- in fact, a much more general and involved statement (see Theorem~\ref{partb}) -- encompasses the topology needed and part (c) follows by applying the argument for part (b) to a combinatorial description of any given finite set. Part (d) for $n\geq 6$ follows by part (c) together with a result which roughly says that mapping degree sets in lower dimensions are mapping degree sets in higher dimensions. The cases $n=4,5$ of part (d) require a different approach to be able to apply parts (b) and (c).

Unless otherwise stated, all manifolds are closed, oriented and connected.

\subsection*{Acknowlegments}
I would like to thank Athanase Papadopoulos and Sumio Yamada for their invitation to speak at the Nisyros Geometry Conference 2025, as well as to write this survey for the book {\em Nisyros 2025: Geometry, History and Philosophy}. This text comprises notes from several lectures I delivered on this topic, mostly related to those at the Max Planck Institute for Mathematics Oberseminar, the Nisyros Geometry Conference and at CUNY Graduate Center. I thank the organisers for their invitations and hospitality. Partial support by a start-up fund from the University of Cyprus, as well as by the research grant EXCELLENCE/0925-{\em Hopf problem and monotone invariants} under the program “THALIA 2021-2027”, co-funded by the European Union through the Research and Innovation Foundation of Cyprus, is gratefully acknowledged.

\section{First examples and motivation}\label{examples}

Given $M$ and $N$, the first natural question arising with the definition of the set $D(M,N)$ is the following:

\begin{prob}
Compute or estimate $D(M,N)$.
\end{prob}

Understanding $D(M,N)$ for various classes of $M$ and $N$ could guide us to a better understanding of the Realization Problem \ref{p:realization}.

\medskip

Clearly, $D(M,N)$ contains zero. The first example that gives us a complete computation in all dimensions is the $n$-sphere $S^n$:

\begin{ex}\label{ex:sphere}
$D(S^n,S^n)=\Z$.
\end{ex}

With a little care, one then obtains the next computation:

\begin{ex}\label{ex:pinch}
$D(M,S^n)=\Z$ for any $n$-manifold $M$. Indeed, observe that $M\simeq M\#S^n$ and pinch to a point $M$. This defines a degree one map, $p\colon M\#S^n\to S^n$, called {\em pinch map}. Then compose with the maps from Example \ref{ex:sphere}
\end{ex}

Having achieved the set $\Z$ using any manifold $M$, now one would like to see the other end, i.e. the set $\{0\}$, as a mapping degree set. This leads us to a first obstruction\index{obstruction}:

\begin{lem}\label{l:fund}
If $f\colon M\to N$ is a map of non-zero degree, then the induced map on the fundamental groups $f_*\colon\pi_1(M)\to\pi_1(N)$ surjects onto a finite index subgroup of $\pi_1(N)$, i.e. $[\pi_1(N):f_*(\pi_1(M))]<\infty$.
\end{lem}

For a proof see for example~\cite[Lemma 9.2.9]{Ne3}. In particular we have the following:

\begin{ex}
Let $n>1$. If $M$ is an $n$-manifold with infinite fundamental group, then $D(S^n,M)=\{0\}$.
\end{ex}

With the two extreme cases $\{0\}$ and $\Z$ having been quickly realised as mapping degree sets, we will try to understand how some proper subsets of $\Z$ occur as such sets. Since $D(S^1,S^1)=\Z$, let us look in dimension two: 
Let $\Sigma_g$ be a surface of genus $g$. 
When $g=0$, i.e. $\Sigma_0=S^2$, then $D(S^2,\Sigma_g)$ and $D(\Sigma_g, S^2)$ have been computed above for any $g$, hence we can assume $g\geq 1$ for either the source manifold or the target manifold. When $g=1$, i.e. $\Sigma_1=T^2=S^1\times S^1$, then:
\begin{itemize}
\item $D(\Sigma_g,T^2)=\Z$. This follows by the same argument as in Example \ref{ex:pinch}, since the 2-torus $T^2$ admits self-maps of any degree and then one can compose with the degree one pinch map $\Sigma_g=\Sigma_{g-1}\#T^2\to T^2$.
\item $D(T^2,\Sigma_g)=\{0\}$, if $g\geq2$. While one can apply again Lemma \ref{l:fund}, an alternative way is via homology, since $\dim H_1(T^2;\Q)=2$ and $\dim H_1(\Sigma_2;\Q)=2g$. This is explained in the lemma below (see~\cite[Lemma 9.2.3]{Ne3} for a proof).
\end{itemize}

\begin{lem}\label{l:homol}
If $f\colon M\to N$ is a map of non-zero degree, then the induced map on rational homology groups $H_i(f;\Q)\colon H_i(M;\Q)\to H_i(N;\Q)$ is surjective.
\end{lem}

We are left with the sets $D(\Sigma_g,\Sigma_h)$, where $g\geq h\geq 2$. Then, an obstruction is given by the {\em simplicial volume}\index{simplicial volume}: Given a topological space $X$ and a homology class $\alpha\in H_n(X;\R)$, Gromov~\cite{Gr1} introduced 
the $\ell^1$-semi-norm of $\alpha$ to be
\[
\|\alpha\|_1:=\inf \biggl\{\sum_{j} |\lambda_j| \ \biggl\vert  \ \sum_j \lambda_j\sigma_j\in C_n(X;\R) \ \text{is a singular cycle representing } \alpha  \biggl\}.
\]
When $M$ is an $n$-dimensional manifold, the simplicial volume of $M$ is given by $\|M\|:=\|[M]\|_1$.

The Gromov norm satisfies the following straightforward property:

\begin{lem}[Functoriality]\label{l:functorial}
Let $X,Y$ be topological spaces and $\alpha\in H_n(X;\R)$. If $f\colon X\to Y$ is a continuous map, then $\|\alpha\|_1\geq \|H_n(f)(\alpha)\|_1$.
\end{lem}

Since for an aspherical surface\index{aspherical!surface} $\Sigma_g$ the simplicial volume is given by (see~\cite{Gr3})
\[
\|\Sigma_g\|=2|\chi(\Sigma)|=4(g-1),
\]
we obtain the following bound\index{mapping degree!bound} on the degree of a map $f\colon\Sigma_g\to\Sigma_h$, $h\geq2$:
\[
|\deg(f)|\leq\frac{g-1}{h-1}.
\]
It is then easy to show that actually every integer that falls within this bound can be realised as a mapping degree of a map from $\Sigma_g$ to $\Sigma_h$:\index{mapping degree! between surfaces}

\begin{prop}\label{p:hsur}
If $g\geq h\geq 2$, then
$$
D(\Sigma_g,\Sigma_h)=\biggl\{0,\pm1,...,\pm\biggl[\frac{g-1}{h-1}\biggl]\biggl\}.
$$
\end{prop}

Proposition \ref{p:hsur} gives us in particular the following:

\begin{ex}\label{ex:surfaces}
Let $k\geq 1$. Then the set $\{0,\pm1,...,\pm k\}$ can be realised as the mapping degree set 
$D(\Sigma_{k+1},\Sigma_2)$.
\end{ex}

Interpretating the simplicial volume of an aspherical surface in terms of the Euler characteristic\index{Euler characteristic}, together with a question of Gromov whether the non-vanishing of the Euler characteristic for an aspherical manifold implies the non-vanishing of the simplicial volume of that manifold (see~\cite{Gr2}), we are led to the following natural question:\index{Euler characteristic!monotonicity}\index{problem!monotonicity}

\begin{prob}[Functoriality of the Euler characteristic for aspherical manifolds]
Is the absolute value of the Euler characteristic functorial for aspherical manifolds? That is, if $M,N$ are aspherical manifolds and $f\colon M\to N$ is a map of non-zero degree, does it hold $|\chi(M)|\geq\deg(f)|\chi(N)|$?
\end{prob}

Example \ref{ex:surfaces} gives us already infinitely many mapping degree sets starting with any integer $k$. If one would like to create as well some gaps between the mapping degrees, then quotients of $S^n$ are suitable:

\begin{ex}
Let $N$ be a spherical quotient $S^n/\pi_1(N)$, and denote by $|\pi_1(N)|$ the order of its fundamental group. Then for any $n$-manifold $M$ we have 
\[
D(M,N)=\{d+|\pi_1(N)|\Z \ |\ \text{for some integers} \ d \ \text{such that} \ 1\leq d\leq |\pi_1(N)|\}.
\]
See~\cite{NWW} for a proof.
\end{ex}

Having in hands the above examples, one can start speculating large classes of sets for which Problem \ref{p:realization} has an affirmative solution. Indeed, all the above examples are (unions) of arithmetic progressions: A  finite sequence of integer intervals 
\[
\{[b_i, c_i], i=1,2,...,l\}
\]
is called {\em arithmetic}, if the lengths of all $[b_i, c_i]$ are equal,
and all the differences $b_{i+1}-b_i$ are equal. When $b_i=c_i$, we obtain a usual finite arithmetic progression. In the proof of part (b) we will see a much more general process that implies the realization of finite progressions\index{realization!finite arithmetic sequence}. This process will be instrumental for parts (c) and (d).

\section{Proof of Theorem \ref{t:main}}

\subsection{Proof of part (a)}

This is an immediate consequence of the fact that there are countably many homotopy types of $n$-manifolds~\cite[Corollary, p.93]{Ma} and thus countably many mapping degree sets $D(M,N)$, where $M,N$ run all over $n$-manifolds. Indeed, the set $\Z-\{0\}$ has uncountably many subsets and countably many finite sets, so it has uncountably many infinite subsets.  In particular, $\Z$ has uncountably many infinite subsets containing zero, and so there are uncountably many such sets that cannot be written in the form $D(M,N)$.

\subsection{Proof of part (b)}\label{s:proofrealizability}

Our main result is the following:

\begin{thm}\cite[Theorem 3.1]{NWW}\label{partb}
For any $k\in \N$ and any integers 
\[
d_1,d_2,...,d_k>0 \ \text{and} \
n_1,n_1',n_2,n_2',...,n_k,n_k'\ge 0,
\] 
there exist $3$-manifolds $M,N$ such that 
\[
D(M,N)=\{d\in\mathbb{Z} \ | \ d=\sum_{i=1}^{k}{m_id_i},\,\,-n_i'\le m_i \le n_i\}.
\]
\end{thm}

Indeed, Theorem \ref{partb} implies Theorem \ref{t:main} (b): Let 
$$\{[b_i, c_i], i=1,2,...,l\}$$ 
be a finite arithmetic sequence of integer  intervals,
where $b_i\le c_i< b_{i+1}$, and $0\in [b_k, c_k]$ for some $1\le k\le l$.
Set
\[
n_1=c_k,\,\,  n_1'=-b_k,\,\, d_2=b_2-b_1,\,\, n_2=l-k,\,\, n_2'=k-1.
\]
 Since $\{b_i, \ i=1,...,l\}$ is an arithmetic sequence with constant difference $d_2$, we have 
$$b_i=b_k+d_2(i-k)=-n_1'+d_2(i-k).$$
Similarly, 
$$c_i=c_k+d_2(i-k)=n_1+d_2(i-k).$$
Thus, we obtain
\begin{equation*}
\begin{aligned}
A=\bigcup_{i=1}^{l}{[b_i,c_i]}
&=\bigcup_{i=1}^{l}[-n_1'+d_2(i-k),n_1+d_2(i-k)]\\
&=\bigcup_{j=1-k}^{l-k}[-n_1'+d_2j,n_1+d_2j]\\
&=\bigcup_{j=-n_2'}^{n_2}[-n_1'+d_2j,n_1+d_2j]\\
&=\bigcup_{j=-n_2'}^{n_2}\{d\in\Z \ | \ d=m_1+jd_2,\,\,-n_1'\le m_1\le n_1\}\\
&=\bigcup_{m_2=-n_2'}^{n_2}\{d\in\Z \ | \ d=m_1+m_2d_2,\,\,-n_1'\le m_1\le n_1\}\\
&=\{d\in\Z \ | \ d=m_1+m_2d_2,\,\,-n_i'\le m_i\le n_i\}.
\end{aligned}
\end{equation*}
Now the proof for Theorem \ref{t:main} (b) follows by Theorem \ref{partb} for $k=2$ and $d_1=1$.

\medskip

We will now sketch the proof of Theorem \ref{partb}.  The building blocks for our construction will be circle bundles over surfaces\index{aspherical!circle bundle}. Given such a bundle $S^1\to K\to\Sigma$, where $\Sigma$ is a surface, the Euler number of $K$ is defined by the Kronecker product
\[
\hat{e}(K)=\langle e(K),[\Sigma]\rangle,
\]
where $e(K)\in H^2(\Sigma;\Z)=\Z$ denotes the Euler class of $K$.

The following lemma determines the mapping degree sets when running over all Euler numbers for a fixed hyperbolic surface. 

\begin{lem}\cite[Lemma 3.4]{NWW}.\label{l:degreecircle}
Let $\Sigma$ be a hyperbolic surface and $K_i\longrightarrow\Sigma$ be the circle bundle with Euler number $\hat{e}(K_i)=i$. Then
\begin{equation}\label{eq.mappingcircle}
D(K_i,K_j)= \left\{\begin{array}{ll}
       \{0,\frac{j}{i}\}, & \text{if} \ i\mid j\\
       \text{} & \\
        \{0\}, & \text{if} \ i\nmid j.
        \end{array}\right.
\end{equation}
\end{lem}

Lemma \ref{l:degreecircle} has its roots in~\cite{Ne2}, where the following result (stated here in a more detailed form) was proved:

\begin{thm}\cite[Theorem 1.3]{Ne2}\label{t:NeHopf}
Suppose $B$ is an aspherical manifold with hyperbolic fundamental group and $M\to B$ is a circle bundle. $M$ admits a self-map of degree greater than one if and only if $M$ is finitely covered by a product $S^1\times \overline B$, where $\overline B$ is a finite cover of $B$.
\end{thm}

Moreover, if a map $f\colon M\to M$ with $|\deg(f)|>1$ exists, where $M$ is as in Theorem \ref{t:NeHopf}, then $f$ is in all (known) cases homotopic to a covering. If $\deg(f)=\pm1$, then $f$ is a homotopy equivalence. 

A long-standing problem attributed to Hopf is whether every self-map of any manifold (not necessarily aspherical) of absolute degree one is a homotopy equivalence~\cite[Problem 5.26]{Kirby}. Theorem \ref{t:NeHopf} and all established results thus far lead us to the following:\index{problem!Hopf}\index{aspherical!strong version of Hopf's problem}

\begin{prob}[Strong version of the Hopf problem for aspherical manifolds]\cite[Problem 1.2]{Ne2}
Is every self-map of non-zero degree of an aspherical manifold either a homotopy equivalence or homotopic to a covering?
\end{prob}

Note that all of the non-zero degree maps in Lemma \ref{l:degreecircle} are homotopic to coverings.

\medskip

Lemma \ref{l:degreecircle} gives us in particular that every two-point set $\{0,d\}$, $d\neq0$, is a mapping degree sets between two circle bundles over a hyperbolic surface. For example:

\begin{ex}\label{ex:surfaces2}
Let $d\in\Z-\{0\}$. Then the set $\{0, d\}$ is the mapping degree set $D(K_1,K_d)$.
\end{ex}

It is now tempting to combine arbitrary sets of the form $\{0,d\}$ to realise finite sets. The next lemma is indeed another basic ingredient for the proof of Theorem \ref{partb}. Given sets of integers $A_i$, $i=1,...,k$, the sum of $A_i$ is defined to be
\[
\sum_{i=1}^k A_i =\biggl\{\sum _{i=1}^k a_i \,| \,a_i\in A_i\biggl\}.
\]
When $A_1,...,A_k$ are equal to the same $A$, we shall denote $\sum_{i=1}^k A_i$ by $\sum^k A$.\index{mapping degree!connected sum}

\begin{lem}\cite[Lemma 3.5]{NWW}\label{l:degreeconnected}
Let $M_1,M_2$ and $N$ be $n$-manifolds. Then
\begin{equation}\label{eq.connected}
D(M_1,N)+D(M_2,N)\subseteq D(M_1\#M_2,N),
\end{equation}
with equality if $\pi_{n-1}(N)=0$. 
\end{lem}

Our building factors -- circle bundles over hyperbolic surfaces -- are aspherical 3-manifolds, in particular $\pi_2=0$, and so Lemma \ref{l:degreeconnected} is indeed applicable. We are now ready to conclude the proof of Theorem \ref{partb}:\index{realization!finite arithmetic sequence}

\begin{proof}[Proof of Theorem \ref{partb}]
Set 
\[
d'=d_1d_2...d_k 
\]
and
\[
d_i'=d'/d_i, \ i=1,...,k.
\]
Consider the following circle bundles over a hyperbolic surface $\Sigma$ given by
\[
N=K_{d'}, \ M_i=K_{d_i'} \ \text{and} \ M_i'=K_{-d_i'}
\]
 with Euler numbers
\[
\hat{e}(N)=d', \ \hat{e}(M_i)=d_i' \ \text{and} \ \hat{e}(M_i')=-d_i'
\]
respectively.

Since $d'/d'_i=d_i$, Lemma \ref{l:degreecircle} tells us that
\begin{equation}\label{eq.d_i}
D(M_i,N)=D(K_{d'_i}, K_{d'})=\{0,d_i\}.
\end{equation}
Similarly,
\begin{equation}\label{eq.-d_i}
D(M_i',N)=\{-d_i,0\}.
\end{equation}

Define  
\begin{equation*}
M=\#_{i=1}^{k}((\#_{n_i}{M_i})\#(\#_{n_i'}{M_i'})).
\end{equation*}
Since $N$ is aspherical, in particular $\pi_2(N)=0$, we can apply  Lemma \ref{l:degreeconnected} successively to compute

\[
D(M,N)=\sum_{i=1}^{k}{(\sum_{j_i=1}^{n_i}{D(M_i,N)}+\sum_{j_i=1}^{n_i'}D(M_i',N))}.
\]
By (\ref{eq.d_i}) and (\ref{eq.-d_i}),
$$\sum_{j_i=1}^{n_i}{D(M_i,N)}+\sum_{j_i=1}^{n_i'}{D(M_i',N)}$$ is the sum of  $n_i$ copies of $\{0, d_i\}$ and of $n'_i$ copies of $\{0, -d_i\}$.
Hence,
\[
\sum_{j_i=1}^{n_i}{D(M_i,N)}+\sum_{j_i=1}^{n_i'}{D(M_i',N)}=\{m_id_i \ |  -n_i'\le m_i\le n_i\}.
\]
We conclude that
$$D(M,N)=\{d\in\Z \ | \ d=\sum_{i=1}^{k}{m_id_i},\,\,-n_i'\le m_i\le n_i\},$$
finishing the proof of Theorem \ref{partb}.
\end{proof}

\subsection{Proof of part (c)}

After proving Theorem \ref{partb}, a natural step would be to consider other types of sequences, such as
\[
0,1,d,d^2,...,d^l,
\] 
for some integer $d>0$. 

This was achieved in~\cite{NWW} with $3l$-manifolds, by taking products of $3$-manifolds:

\begin{thm}\cite[Theorem 4.1]{NWW}\label{productdegrees}
Given integers $1\leq d_1 \leq d_2 \leq\cdots \leq d_l$, there exist $3l$-manifolds $M,N$ such that 
\[
D(M,N)=\{0,1\}\cup\biggl\{\prod_{j\in S}d_j \ | \ \emptyset\ne S\subseteq \{1,2,...,l\}\biggl\}.
\]
\end{thm}

The case $l=1$ of Theorem \ref{productdegrees} is crucial for its proof:

\begin{thm}\label{1,k}
For any integer $d>1$, there exist $3$-manifolds $Q$ and $P$ such that $D(Q,P)=\{0,1,d\}$. 
\end{thm}

A straightforward, but still essential, observation for the proof of Theorem \ref{1,k} is the following:\index{mapping degree!connected sum}

\begin{lem}\cite[Lemma 4.3]{NWW}\label{connectedsumdegrees}
Let $M,N_1$ and $N_2$ be $n$-manifolds. Then 
\begin{equation}\label{eq:conntarget}
D(M,N_1\#N_2)\subseteq D(M,N_1)\cap D(M,N_2).
\end{equation}
\end{lem}

The content of the proof of Theorem \ref{1,k} is fundamental for the proof of part (c) of Theorem \ref{t:main} and thus we provide it for completeness:

\begin{proof}[Proof of Theorem \ref{1,k}]
Let $q>d$ be a prime number, and consider the following connected sums, where, as in Section \ref{s:proofrealizability}, $K_i$ denotes the $S^1$-bundle over a fixed hyperbolic surface with Euler number $i$:
\[
Q=(\#_d K_q)\# K_d\#K_{d^2} 
 \]
 and
 \[
 P= K_q\#K_{d^2}.
 \]
 Let $Q_1=(\#_dK_q)\#K_d$. By Lemma \ref{l:degreecircle}, $K_d$ is a $d$-fold covering of $K_{d^2}$, and so we obtain a covering 
 \begin{equation}\label{eq.d-fold}
Q_1=(\#_dK_q)\#K_d\to K_q\#K_{d^2}=P
\end{equation}
of degree $d$. 
Note that
\[
Q=P\#(\#_{d-1}K_q)\#K_{d}=
Q_1\#K_{d^2}.
\]
By (the proof of) Lemma \ref{connectedsumdegrees}, there are degree one maps $Q\to Q_1$ and $Q\to P$. Together with (\ref{eq.d-fold}), we deduce 
\begin{equation}\label{eq.inclusion1}
\{0,1,d\}\subseteq D(Q,P).
\end{equation}

We will now show the converse inclusion. Lemma \ref{connectedsumdegrees} implies that
\begin{equation}\label{eq.subsetintersection}
 D(Q,P)\subseteq D(Q,K_q)\cap D(Q,K_{d^2}).
 \end{equation}
Since $K_q$ is aspherical, in particular $\pi_2(K_q)=0$,  Lemma \ref{l:degreeconnected} implies that
 \[
 D(Q,K_q)=\sum^{d}D(K_q,K_q)+D(K_d,K_q)+D(K_{d^2},K_q).
 \]
Since $d$ and $q$ are coprime, Lemma \ref{l:degreecircle} implies that
\[
D(K_q,K_q)=\{0,1\},
\]
\[D(K_d,K_q)=D(K_{d^2},K_q)=\{0\},
\] 
and thus 
\begin{equation}\label{eq.QKq}
D(Q, K_q)=\{0,1,...,d\}. 
\end{equation}
Applying the same argument we obtain
\begin{equation}\label{eq.QKd^2}
D(Q,K_{d^2})=\{0,1,d,d+1\}. 
\end{equation}
Then by (\ref{eq.subsetintersection}), (\ref{eq.QKq}) and (\ref{eq.QKd^2}) we have
\begin{equation}\label{eq.inclusion2}
D(Q,P)\subseteq\{0,1,...,d\}\cap\{0,1,d,d+1\}=\{0,1,d\}.
\end{equation}

The theorem follows by (\ref{eq.inclusion1}) and (\ref{eq.inclusion2}).
\end{proof}

Now, with Theorem \ref{1,k} in hands, one can apply generalisations of results of~\cite{Ne1} about mapping degree sets between products and conclude the proof of Theorem \ref{productdegrees}; we refer the reader to the last section of~\cite{NWW} for the details, and continue with the proof of part (c) of Theorem \ref{t:main}.

\medskip

Clearly, Lemma \ref{connectedsumdegrees} holds with any sequence of manifolds $N_1,...N_k$, i.e.,
\[
D(M,\#_{i=1}^kN_i)\subseteq \bigcap_{i=1}^kD(M,N_i).
\]
However, equality fails in general as illustrated by the following examples:


\begin{ex}\label{ex:connectedsumcounter}\
\begin{itemize}
\item[(i)] We saw in Section \ref{examples} that $D(T^2,T^2)=\Z$ and $D(T^2,T^2\#T^2)=\{0\}$.
\item[(ii)] There is a 2-fold covering $S^3\to\R P^3$, however any map $S^3\to\R P^3\#\R P^3$ has degree zero by Lemma \ref{l:fund}.
\end{itemize}
\end{ex}

Example \ref{ex:connectedsumcounter} indicates furthermore that equality cannot be achieved even if one is willing to glue maps and consider $D(M_1\#M_2,N_1\#N_2)$ on the left-hand side of \eqref{eq:conntarget}. Nevertheless, gluing is possible if one deals with $\pi_1$-surjective maps, and thus we have the following:\index{mapping degree!connected sum}

\begin{lem}\cite{RW,SWWZ}\label{l:WR}
Let $f_i\colon M_i\to N_i$, $i=1,...,k$, be $\pi_1$-surjective maps of degree $d$ between $n$-manifolds, $n\geq3$. Then there is a $\pi_1$-surjective map $f\colon \#_{i=1}^kM_i\to \#_{i=1}^kN_i$ of degree d.
\end{lem}

Combining the maps of Lemma \ref{l:degreecircle} with Lemma \ref{l:WR} in order to realise more degrees cannot be done directly, since the maps of Lemma \ref{l:degreecircle} are covering maps and thus not $\pi_1$-surjective. Hence, one needs to manipulate the gluing of these maps with some additional structure on the domain so that $\pi_1$-surjectivity is achieved.  But then one needs to be careful so that the new manifold does not introduce new mapping degrees. 
A promising modification is to consider free products with free groups at the level of fundamental groups, and thus connected sums with copies of $S^{n-1}\times S^1$ at the topological level (in the domain). This is especially intriguing for the 3-manifold case, since any map from $S^2\times S^1$ to the manifolds of Lemma \ref{l:degreecircle} has degree zero (recall that the manifolds of Lemma \ref{l:degreecircle} are aspherical). Indeed, the following proposition shows the general case using Lemmas \ref{l:degreeconnected} and \ref{connectedsumdegrees}:\index{mapping degree!connected sum}

\begin{prop}\cite[Prop. 3.7]{CMV}\label{CMV}
Let $M_i$, $N_i$, $i = 1, . . . , r$, be $n$-manifolds, $n\geq 3$ so that
\begin{itemize}
\item[(1)] $\pi_{n-1}(N_i) = 0$ for $i = 1,...,r$;
\item[(2)] $D(M_i,N_j) = \{0\}$ for $i\neq j$;
\item[(3)] $D(M_i,N_i)\cap D(M_j,N_j)$ is finite for $i\neq j$.
\end{itemize}
Then there exists an integer $l \geq 0$ such that
\[
D((\#_{j=1}^rM_j)\#(\#^l(S^{n-1}\times S^1)),\#_{i=1}^rN_j)={\bigcap_{i=1}^r}D(M_i,N_i).
\]
\end{prop}

Hence, one wishes to express an arbitrary finite set as an intersection of sums of sets $\{0,k\}$. This was done in \cite{CMV}:
Given a finite set $B$ of non-zero integers, set
\[
S_B=\sum_{b\in B}\{0,b\}.
\]

\begin{ex}
Let $B=\{2,2,3\}$. Then
\[
S_B=\{0,2\}+\{0,2\}+\{0,3\}=\{0,2,3,4,5,7\}.
\]
\end{ex}

The main arithmetic combinatorial result in \cite{CMV} is the following:

\begin{prop}\cite[Prop. 2.2]{CMV}\label{p:CMV}
Let $d_1, . . . , d_n$ be pairwise distinct non-zero integers. There exist finite sequences $B_i$, $i = 0, . . . , n$, of non-zero integers, such that
\[
\{0,d_1,...,d_n\}=\bigcap_{i=0}^nS_{B_i}.
\]
\end{prop}

We illustrate Proposition \ref{p:CMV} below and refer the reader to~\cite{CMV} for the proof of this proposition and for more details on the choice of the sets $B_i$.

\begin{ex}
Let $d_1=-2$, $d_2=3$. We have $-2<0<3$. Setting 
$
B_0=\{-1,-1,1,1,1\}
$
we obtain
\[
S_{B_0}=[-1,3]\cap\Z.
\]
Then, we have $\max\{3,3+2\}=5$ and set 
$
B_0=\{-1,-1,-1,-1,7\}.
$
We picked $7$ as a prime greater than $5$; also note that we have four $-1$'s because $4=7-3$. We compute
\[
S_{B_1}=([-4,0]\cup[3,7])\cap\Z.
\]
Finally, we have 
$\max\{2,2+3\}=5$ and set 
$
B_2=\{1,1,1,1,1,-7\}; 
$
again note that $5=7-3$ which is the number of $1$'s. Computing, we obtain
\[
S_{B_2}=([-7,-2]\cup[0,5])\cap\Z.
\]
The intersection of the $S_{B_i}$'s is then
\[
\bigcap_{i=0}^2S_{B_i}=\{0,-2,3\}.
\]
\end{ex}

\index{realization!finite set between $3$-manifolds}
We are now ready to complete the proof of part (c) of Theorem \ref{t:main} (Theorem A in~\cite{CMV}): Let $d_1,d_2,...,d_n$ be non-zero pairwise distinct integers. By Proposition \ref{p:CMV}, we can write
\[
\{0,d_1,...,d_n\}=\bigcap_{i=0}^nS_{B_i},
\]
for some finite sequences $B_i$, $i = 0, . . . , n$, of non-zero integers.

For each $i=0,1,...,n$, pick a prime $q_i$ such that
$$
q_i>\max_{b\in B_i}\{|b|\}
$$
and let
\[
\alpha_i=q_i\prod_{b\in B_i}b, \ i=0,1,...,n.
\]
Setting
\[
M_i=\#_{b\in B_i}K_{\alpha_i/b} \ \text{and} \ N_i=K_{\alpha_i},
\]
where the $K_{p}$ are the $S^1$-bundles of Lemma \ref{l:degreecircle},
we obtain by Lemma \ref{l:degreeconnected}
\[
D(M_i,N_i)=S_{B_i}, \ i=0,1,...,n
\]
and
\[
D(M_i,N_i)=\{0\} \ \text{for all} \ i\neq j.
\]
Hence, by Proposition \ref{CMV}, there is some $l>0$ such that
\[
D((\#_{i=0}^nM_i)\#(\#_l(S^2\times S^1)),\#_{i=0}^nN_i)=\bigcap_{i=0}^nS_{B_i}=\{0,d_1,...,d_n\}.
\]

\subsection{Proof of part (d)}

\index{realization!finite set between $n$-manifolds, $n\geq4$}

We will now indicate that any finite set containing zero is the mapping degree set between some $n$-manifolds, for each $n\geq4$. Having the case $n=3$ from part (c) of Theorem \ref{t:main}, we can conclude the proof for each $n\geq6$ by the following result:\index{mapping degree!product}

\begin{thm}\cite[Theorem 3.3]{NSTWW1}\label{t:NSTWW}
For any $n$-manifolds $M$ and $N$ and any integer $k\geq3$, there are infinitely many hyperbolic $k$-manifolds $W$ such that
\[
D(M,N)=D(M\times W, N\times W).
\]
\end{thm}

Clearly, Theorem \ref{t:NSTWW} cannot be used to prove the cases $n=4,5$ of part (d) of Theorem \ref{t:main}. In the remainder of this paragraph, we will summarize briefly the proof for these two dimensions and refer the reader to~\cite{NSTWW2} for the details.

\medskip

As in the case of $3$-manifolds, circle bundles will play again an instrumental role as building factors. The main contribution of~\cite{NSTWW2} concerns understanding maps of non-zero degree between circle bundles in a certain general setting in any dimension motivated by the $3$-dimensional case.

Given two circle bundles $E_1$ and $E_2$ of the same dimension, we define the {\em fiber-preserving  mapping degree set} as
\[
D_{FP}(E_1,E_2):=\{d\in\Z \ |\, \text{$\exists$ fiber-preserving map} \  f\colon E_1\to E_2 \ \text{such that}\ \deg( f)=d\}.
\]

The next result describes the mapping degree sets of fiber-preserving maps:

\begin{thm}\cite[Theorem 1.2]{NSTWW2}\label{main1} Let  $M_i$ be an $n$-manifold and let 
$E_i \to M_i$ be an $S^1$-bundle with Euler class $e_i\in H^2(M_i;\mathbb{Z})$, $i=1,2$. Then 
$$D_{FP}(E_1, E_2)=\{0\} \cup\{k\cdot \deg(f)\ |\  k\ne 0,  f\colon M_1\to M_2, \deg(f)\ne 0 \ \text{and}\ H^2(f)(e_2)=ke_1\},$$
where $H^2(f)\colon H^2(M_2;\mathbb{Z})\to H^2(M_1;\mathbb{Z})$ denotes  the induced homomorphism on cohomology.
\end{thm}

Hence, one would gain a very good understanding of mapping degree sets between circle bundles if one is able to find a fiber-preserving map in the homotopy class of a non-zero degree map. This is the main result of~\cite{NSTWW2} in the following group theoretic setting:

\begin{thm}\cite[Theorem 1.1]{NSTWW2}\label{main2}
Let  $E_i$ be an $S^1$-bundle over an $n$-manifold $M_i$ with $i=1,2$, and assume that $M_2$ is aspherical. Let $ f\colon   E_1\to  E_2$ be a map.
If the induced map $f_*$ on $\pi_1$ sends the fiber subgroup of $\pi_1(E_1)$ to the fiber subgroup of $\pi_1(E_2)$, then there is a fiber-preserving map in the homotopy class of $f$.
\end{thm}

From the above theorem we conclude the following:

\begin{cor}\cite[Corollary 4.1]{NSTWW2}\label{fiber3}
For $i=1,2$, let 
$E_i$ be an $S^1$-bundle over an $n$-manifold $M_i$, where $M_2$ is aspherical, and let $f\colon   E_1\to  E_2$ be a map of non-zero degree.
If each finite index subgroup of  $\pi_1(M_2)$ is center-free,
 then $f$ is homotopic to a fiber-preserving map. As a consequence, we have 
$$D_{FP}(E_1 ,E_2)=  D(E_1, E_2).$$
\end{cor}

With the strong machinery provided by the above results, we would now like to apply the same argument as in the proof of part (c) of Theorem \ref{t:main}. Below, we will denote a circle bundle over an $n$-manifold $M$ with Euler class  $a\in H^2(M; \Z)$ by $\tilde M_a$. What we need is the following result in dimensions four and five:

\begin{thm}\cite[Theorem 1.4]{NSTWW2}\label{0,k}
For each non-zero integer $k$ and  $n=4, 5$,  there exist aspherical  $n$-manifolds of the form $\tilde N_b,\tilde N_{kb}$ such that  
$$D(\tilde N_b,\tilde N_{kb})=\{0, k\}.$$
 \end{thm}

The latter theorem is a consequence of the next two results. First, there exist manifolds (the bases of the desired circle bundles) with the following properties:

\begin{prop}\cite[Prop. 4.4]{NSTWW2}\label{0,k1}
 In each of the dimensions three and four, there exists a manifold $N$ such that:
\begin{itemize}
\item[(i)] $N$ is aspherical and every finite-index subgroup of  $\pi_1(N)$ is center-free, 
\item[(ii)] $D(N, N)=\{0,1\}$, 
\item[(iii)] in dimension three, there is a non-torsion class $b\in H^2(N;\mathbb{Z})$ such that $H^2(f)(b)=b$ for any degree one map $f\colon N\to N$; in dimension four, $N$ is hyperbolic with positive second Betti number and isometry group of odd order.
\end{itemize}
 \end{prop}
 
 Indeed, with the manifolds given by Proposition \ref{0,k1} we obtain Theorem \ref{0,k} by the following general result:

 \begin{prop}\cite[Prop. 4.3]{NSTWW2}\label{appl2}
 Suppose $N$ is an $n$-manifold satisfying (i)--(iii) in Proposition \ref{0,k1}, where $b\in H^2(N;\mathbb{Z})$ is not torsion.  Then for any non-zero integers $k$ and $m$, we have
 $$
D(\tilde N_{mb},\tilde N_{kb})= \left\{\begin{array}{ll}
       \{0,k/m\}, & \text{if} \ m\mid k \\
       \text{} & \\
        \{0\}, & \text{if} \ m\nmid k.
        \end{array}\right.
$$
\end{prop}

\bibliographystyle{amsalpha}

\end{document}